\documentclass{amsart}
\usepackage{graphics,verbatim, amsmath, amssymb, amsthm, amsfonts, epsfig, psfrag}	

\newtheorem{proposition}{Proposition}[section]
\newtheorem{theorem}[proposition]{Theorem}

\newtheorem{lemma}[proposition]{Lemma}
\newtheorem{prop}[proposition]{Proposition}

\newtheorem{counterexample}[proposition]{Counter-example}

\theoremstyle{definition}

\theoremstyle{remark}
\newtheorem{remark}[proposition]{Remark}

\newcommand{\newword}[1]{\textbf{\emph{#1}}}

\newcommand{\from}{\leftarrow}

\newcommand{\g}{\mathbf{g}}

\newcommand{\set}[1]{{\left\lbrace #1 \right\rbrace}}
\newcommand{\pidown}{\pi_\downarrow}

\newcommand{\br}[1]{\langle #1 \rangle}

\newcommand{\join}{\vee}

\renewcommand{\Join}{\bigvee}
\newcommand{\Meet}{\bigwedge}

\newcommand{\ck}{^\vee}

\newcommand{\Tits}{\mathrm{Tits}}
\newcommand{\Cone}{\mathrm{Cone}}

\newcommand{\LL}{\mathcal{L}}

\DeclareMathOperator{\inv}{inv}

\title{Sortable Elements for Quivers with Cycles}
\author{Nathan Reading and David E Speyer}
\thanks{The first author was partially supported by NSA grant H98230-09-1-0056.
The second author was funded by a Research Fellowship from the Clay Mathematics Institute.}

\begin{document}

\begin{abstract}
Each Coxeter element $c$ of a Coxeter group~$W$ defines a subset of~$W$ called the $c$-sortable elements.
The choice of a Coxeter element of~$W$ is equivalent to the choice of an acyclic orientation of the Coxeter diagram of~$W$.
In this paper, we define a more general notion of $\Omega$-sortable elements, where~$\Omega$ is an arbitrary orientation of the diagram, and show that the key properties of $c$-sortable elements carry over to the $\Omega$-sortable elements.
The proofs of these properties rely on reduction to the acyclic case, but the reductions are nontrivial; in particular, the proofs rely on a subtle combinatorial property of the weak order, as it relates to orientations of the Coxeter diagram.
The $c$-sortable elements are closely tied to the combinatorics of cluster algebras with an acyclic seed; the ultimate motivation behind this paper is to extend this connection beyond the acyclic case.
\end{abstract}

\maketitle

\section{Introduction}\label{intro}
The results of this paper are purely combinatorial, but are motivated by questions in the theory of cluster algebras.
To define a cluster algebra, one requires the input data of a skew-symmetrizable integer matrix;
that is to say, an $n \times n$ integer matrix $B$ and a vector of positive integers $(\delta_1, \ldots, \delta_n)$ such that $\delta_i B_{ij} = - \delta_j B_{ji}$. 
(For the experts: we are discussing cluster algebras without coefficients.)
This input data defines a recursion which produces, among other things, a set of \newword{cluster variables}.
Each cluster variable is a rational function in $x_1,\ldots,x_n$, and the cluster variables are grouped into overlapping sets of size $n$, called clusters.
The cluster algebra is the algebra generated, as a ring, by the cluster variables.

Experience has shown\footnote{See~\cite{ca2}, \cite{camb_fan} for direct connections between cluster algebras and root systems; see \cite{BMRRT} and~\cite{Keller}, and the works cited therein, for connections between cluster algebras and quivers,  and see, for example,~\cite{Kac} for the relationship between quivers and root systems.}
 that the properties of the cluster algebra are closely related to the properties of the corresponding Kac-Moody root system, coming from the generalized Cartan matrix $A$ defined by $A_{ii}=2$ and $A_{ij} = - |B_{ij}|$ for $i \neq j$.
Let~$W$ stand for the Weyl group of the Kac-Moody algebra.
From the Cartan matrix, one can read off the \newword{Coxeter diagram} of~$W$.
This is the graph $\Gamma$ whose vertices are labeled by $\{ 1,2, \ldots, n \}$ and where there is an edge connecting $i$ to $j$ if and only if $A_{ij} \neq 0$.
To encode the structure of $B$, it is natural to orient $\Gamma$, directing $i \from j$ if $B_{ij}>0$. 
This orientation of $\Gamma$ is denoted by~$\Omega$.

This paper continues a project \cite{cambrian,camb_fan,typefree} of attempting to understand the structure of cluster algebras by looking solely at the combinatorial data $(W, \Gamma, \Omega)$.
In the previous papers, it was necessary to assume that~$\Omega$ was acyclic. 
This assumption is no restriction when $\Gamma$ is a tree---in particular, whenever~$W$ is finite.
In general, however, many of the most interesting and least tractable cluster algebras correspond to orientations with cycles.
Methods based on quiver theory, which have proved so powerful in the investigation of cluster algebras, were originally also inapplicable in the case of cycles; recent work of Derksen, Weyman and Zelevinsky \cite{QWP1} has partially improved this situation.

The aim of this note is to extend the combinatorial results of~\cite{typefree} to the case of an orientation with cycles. 
This paper does not treat cluster algebras at all, but proves combinatorial results which will be applied to cluster algebras in a future paper.
The results can be understood independently of cluster algebras and of the previous papers.
The arguments are valid not only for the Coxeter groups that arise from cluster algebras, but for Coxeter groups in full generality.
In this sense, the title of the paper is narrower than the subject matter, but we have chosen the narrow title as a briefer alternative to a title such as ``Sortable elements for non-acyclic orientations of the Coxeter diagram.''

Let $S$ be the set of simple generators of~$W$, i.e.\ the vertex set of $\Gamma$.
If~$\Omega$ is acyclic, then we can order the elements of $S$ as $s_1,s_2,\ldots,s_n$ so that, if there is an edge $s_i \from s_j$, then $i < j$.  
The product $c(\Omega)=s_1s_2\cdots s_n$ is called a \newword{Coxeter element} of~$W$.
Although~$\Omega$ may not uniquely determine the total order $s_1,s_2,\ldots,s_n$, the Coxeter element $c(\Omega)$ depends only on $\Omega$.  
Indeed, Coxeter elements of~$W$ are in bijection with acyclic orientations of~$\Gamma$.  

Given a Coxeter element $c$, every element~$w$ of~$W$ has a special reduced word called the \newword{$c$-sorting word} of~$w$. 
The \newword{$c$-sortable elements} of \cite{sortable,sort_camb,camb_fan,typefree} are the elements of~$W$ whose $c$-sorting word has a certain special property.
We review the definition in Section~\ref{main}.
Sortable elements provide a natural scaffolding on which to construct cluster algebras \cite{camb_fan,clusters}.
The goal of this paper is to provide a definition of $\Omega$-sortable elements for arbitrary orientations which have the same elegant properties as in the acyclic case (always keeping in mind the underlying goals related to cluster algebras).

Say that a subset~$J$ of $S$ is \newword{$\Omega$-acyclic} if the induced subgraph of $\Gamma$ with vertex set~$J$ is acyclic. 
If~$J$ is $\Omega$-acyclic, then the restriction $\Omega|_J$ defines a Coxeter element $c(\Omega,J)$ for the standard parabolic subgroup $W_J$. 
(Here $W_J$ is the subgroup of~$W$ generated by~$J$.) 
We define~$w$ to be $\Omega$-sortable if there is some $\Omega$-acylic set~$J$ such that~$w$ lies in $W_J$ and~$w$ is $c(\Omega,J)$-sortable, when considered as an element of $W_J$.
The definition appears artificial at first, but in Section~\ref{main} we present an equivalent, more elegant definition of $\Omega$-sortability which avoids referencing the definition from the acyclic case.

When $J$ is $\Omega$-acyclic, we will often regard $\Omega|_J$ as a poset.
Here the order relation, written $\le_J$, is the transitive closure of the relation with $r >_J s$ if there is an edge $r \to s$. 

We now summarize the properties of $\Omega$-sortable elements for general $\Omega$.  
All of these properties are generalizations of results on the acyclic case which were proved in~\cite{typefree}.
As in the acyclic case, we start with a recursively defined downward projection map $\pidown^{\Omega}:W\to W$.
(The definition is given in Section~\ref{main}.)
We then prove the following property of $\pidown^\Omega$.

\begin{prop} \label{Maximal Sortable}
Let $w\in W$.
Then $\pidown^{\Omega}(w)$ is the unique maximal (under weak order) $\Omega$-sortable element weakly below~$w$.
\end{prop}

As immediate corollaries of Proposition~\ref{Maximal Sortable}, we have the following results.

\begin{theorem}\label{o p}
The map $\pidown^{\Omega}$ is order-preserving.
\end{theorem}

\begin{prop}\label{idem} 
The map $\pidown^{\Omega}$ is idempotent (i.e. $\pidown^\Omega\circ\pidown^\Omega=\pidown^\Omega$).
\end{prop}

\begin{prop} \label{downward}
Let $w \in W$. 
Then $\pidown^{\Omega}(w) \leq w$, with equality if and only if~$w$ is $\Omega$-sortable. 
\end{prop}

We also establish the lattice-theoretic properties of $\Omega$-sortable elements and of the map $\pidown^\Omega$. 

 \begin{theorem} \label{Sublattice}
 If $A$ is a nonempty set of $\Omega$-sortable elements then $\Meet A$ is $\Omega$-sortable.
 If $A$ is a set of $\Omega$-sortable elements such that $\Join A$ exists, then $\Join A$ is $\Omega$-sortable.
 \end{theorem}

\begin{theorem} \label{QuotientLattice}
If $A$ is a nonempty subset of~$W$ then $\pidown^{\Omega} \left(\Meet A\right) = \Meet \pidown^{\Omega} A$.
If $A$ is a subset of~$W$ such that $\Join A$ exists, then $\pidown^{\Omega} \left(\Join A\right) = \Join \pidown^{\Omega} A$.
\end{theorem}

None of these results are trivial consequences of the definitions; the proofs are nontrivial reductions to the acyclic case.
Our proofs rely on the following key combinatorial result. 

\begin{prop} \label{Maximal Coxeter}
Let~$w$ be an element of~$W$ and~$\Omega$ an orientation of $\Gamma$. 
Then there is an $\Omega$-acyclic subset $J(w,\Omega)$ of $S$ which is maximal (under inclusion) among those $\Omega$-acyclic subsets $J'$ of $S$ having the property that $w \geq c(\Omega,J')$.
\end{prop}

We prove Proposition~\ref{Maximal Coxeter} by establishing a stronger result, which we find interesting in its own right.
Let $L(w,\Omega)$ be the collection of subsets~$J$ of $S$ such that~$J$ is $\Omega$-acyclic and $c(\Omega,J) \leq w$.

\begin{theorem}\label{antimatroid}
For any orientation~$\Omega$ of $\Gamma$ and any $w\in W$, the collection $L(w,\Omega)$ is an antimatroid.
\end{theorem}
We review the definition of antimatroid in Section~\ref{background}.
By a well-known result (Proposition~\ref{anti max}) on antimatroids, Theorem~\ref{antimatroid} implies Proposition~\ref{Maximal Coxeter}.

A key theorem of~\cite{typefree} is a very explicit geometric description of the fibers of $\pidown^c$ (the acyclic version of $\pidown^{\Omega}$). 
To each $c$-sortable element is associated a pointed simplicial cone $\Cone_c(v)$, and it is shown  \cite[Theorem~6.3]{typefree} that $\pidown^c(w)=v$ if and only if $wD$ lies in $\Cone_c(v)$, where $D$ is the dominant chamber.
The cones $\Cone_c(v)$ are defined explicitly by specifying their facet-defining hyperplanes.
The geometry of the cones $\Cone_c(v)$ is intimately related with the combinatorics of the associated cluster algebra.
(This connection is made in depth in~\cite{clusters}.)
In this paper, we generalize this polyhedral description to the fibers of $\pidown^\Omega$, when~$\Omega$ may have cycles.
We will see that this polyhedral description, while not incompatible with the construction of cluster algebras, is nevertheless incomplete for the purposes of constructing cluster algebras.

We conclude this introduction by mentioning a negative result.
In \cite[Theorem~4.3]{typefree} (cf. \cite[Theorem~4.1]{sortable}), $c$-sortable elements (and their $c$-sorting words) are characterized by a ``pattern avoidance'' condition given by a skew-symmetric bilinear form. 
Generalizing these pattern avoidance results has proved difficult.
In particular, the verbatim generalization fails, as we show in Section~\ref{Alignment sec}. 

The paper proceeds as follows.
In Section~\ref{background}, we establish additional terminology and definitions, prove Theorem~\ref{antimatroid}, and explain how Theorem~\ref{antimatroid} implies Proposition~\ref{Maximal Coxeter}.
In Section~\ref{main}, we give the definitions of $c$-sortability and $\Omega$-sortability, and prove Proposition~\ref{Maximal Sortable} and Theorems~\ref{Sublattice} and~\ref{QuotientLattice}. 
Section~\ref{sec fibers} presents the polyhedral description of the fibers of $\pidown^{\Omega}$.
In Section~\ref{Alignment sec}, we discuss the issues surrounding the characterization of $\Omega$-sortable elements by pattern avoidance.

In writing this paper, we have had to make a number of arbitrary choices of sign convention.
Our choices are completely consistent with our sign conventions from~\cite{typefree} and are as compatible as possible with the existing sign conventions in the cluster algebra and quiver representation literature. 
Our bijection between Coxeter elements and acyclic orientations of $\Gamma$ is the standard one in the quiver literature, but is opposite to the convention of the first author in~\cite{sortable}.
We summarize our choices in Table~\ref{Conventions}.

\renewcommand{\arraystretch}{1.25}
\begin{table}[h]
\begin{center}
\begin{tabular}{| l l  l |}
\hline
\multicolumn{3}{| l |}{For $i \neq j$ in $[n]$, the following are equivalent:} \\
\quad \quad & There is an edge of $\Gamma$ oriented & $s_i \from s_j$. \\
\quad \quad  & The $B$-matrix of the corresponding cluster algebra has & $B_{ij} = - A_{ij} > 0$. \\ \hline
\multicolumn{3}{| l |}{If $J \subseteq [n]$ is $\Omega$-acyclic and $i \neq j$ are in $J$, the following are equivalent:} \\
\quad \quad & There is an oriented path in $J$ of the form & $i \from \cdots \from j$. \\
\quad \quad & In the poset $\Omega|_J$, we have & $i <_J j$. \\
\quad \quad & All reduced words for $c(\Omega, J)$ are of the form & $\cdots s_i \cdots s_j \cdots$. \\ \hline
\end{tabular}
\end{center}
\smallskip
\caption{Sign Conventions} \label{Conventions}
\end{table}

\section{Coxeter groups and antimatroids} \label{background}
We assume the definition of a Coxeter group~$W$ and the most basic combinatorial facts about Coxeter groups.
Appropriate references are \cite{BjBr,Bourbaki,Humphreys}.
For a treatment that is very well aligned with the goals of this paper, see \cite[Section~2]{typefree}.
The symbol $S$ will represent the set of defining generators or \newword{simple generators} of~$W$.
For each $s,t\in S$, let $m_(s,t)$ denote the integer (or $\infty$) such that $(st)^{m(s,t)} = e$.  
The Coxeter diagram $\Gamma$ of~$W$ was defined in Section~\ref{intro}.
We note here that, for $s,t\in S$, there is an edge connecting $s$ and $t$ in $\Gamma$ if and only if $s$ and $t$ fail to commute.
(The usual edge labels on $\Gamma$, which were not described in Section~\ref{intro}, are not necessary in this paper.)
For $w \in W$, the \newword{length} of~$w$, denoted $\ell(w)$, is the length of the shortest expression for~$w$ in the simple generators.
An expression which achieves this minimal length is called \newword{reduced}.

The \newword{(right) weak order} on~$W$ sets $u \leq w$ if and only if $\ell(u) + \ell(u^{-1} w) = \ell(w)$. 
Thus $u\le w$ if there exists a reduced word for $w$ having, as a prefix, a reduced word for $u$.
Conversely, if $u\le w$ then any given reduced word for $u$ is a prefix of some reduced word for $w$. 
For any $J\subseteq S$, the standard parabolic subgroup $W_J$ is a (lower) order ideal in the weak order on~$W$.
(This follows, for example, from the prefix characterization of weak order and \cite[Corollary~1.4.8(ii)]{BjBr}.)

We need another characterization of the weak order.
We write $T$ for the \newword{reflections} of~$W$.
An \newword{inversion} of $w\in W$ is a reflection $t\in T$ such that $\ell(tw)<\ell(w)$.
Write $\inv(w)$ for the set of inversions of~$w$.
If $a_1\cdots a_k$ is a reduced word for~$w$ then 
\[\inv(w)=\set{a_1, \ a_1a_2a_2, \ \ldots, \ a_1a_2\cdots a_k\cdots a_2a_1},\]
and these $k$ reflections are distinct.  
We will review a geometric characterization of inversions below.
The weak order sets $u\le v$ if and only if $\inv(u)\subseteq\inv(v)$.
As an easy consequence of this characterization of the weak order (see, for example, \cite[Section~2.5]{typefree}), we have the following lemma.

\begin{lemma} \label{AboveBelowLite}
Let $s \in S$. 
Then the map $w\mapsto sw$ is an isomorphism from the weak order on $\set{w\in W:w\not\ge s}$ to the weak order on $\set{w\in W:w\ge s}$. 
\end{lemma}

The weak order is a meet semilattice, meaning that any nonempty set $A\subseteq W$ has a meet.  
Furthermore, if a set $A$ has an upper bound in the weak order, then it has a join.

Given $w\in W$ and $J\subseteq S$, there is a map $w\mapsto w_J$ from~$W$ to $W_J$, defined by the property that $\inv(w_J)=\inv(w)\cap W_J$. 
(See, for example \cite[Section~2.4]{typefree}.)
For $A\subseteq W$ and $J\subseteq S$, let $A_J=\set{w_J:w\in A}$. 
The following is a result of Jedli\v{c}ka~\cite{Jed}.

\begin{prop}\label{para hom} 
For any $J\subseteq S$ and any subset $A$ of~$W$, if $A$ is nonempty then $\Meet (A_J)=\left( \Meet A \right)_J$ and, if $\Join A$ exists, then $\Join (A_J)$ exists and equals $\left( \Join A \right)_J$.
\end{prop}
As an immediate corollary:
\begin{prop}\label{wJ o p}
The map $w\mapsto w_J$ is order-preserving.
\end{prop}

We now fix a reflection representation for~$W$ in the standard way.
For a more in-depth discussion of the conventions used here, see \cite[Sections~2.2--2.3]{typefree}.
We first form a \newword{generalized Cartan matrix} for~$W$.
This is a real matrix~$A$ with rows and columns indexed by $S$ such that:
\begin{enumerate}
\item[(i) ]$A_{ss}=2$ for every $s\in S$;
\item[(ii) ]$A_{ss'}\le 0$ with $\displaystyle A_{ss'}A_{s's}=4\cos^2\left(\frac{\pi}{m(s,s')}\right)$ when $s\neq s'$ and $m(s,s')<\infty$, and $\displaystyle A_{ss'}A_{s's} \ge 4$ if $m(s,s')=\infty$; and 
\item[(iii) ]$A_{ss'}=0$ if and only if $A_{s's}=0$.
\end{enumerate}
The matrix $A$ is \newword{crystallographic} if it has integer entries.
We assume that $A$ is \newword{symmetrizable}.
That is, we assume that there exists a positive real-valued function $\delta$ on $S$ such that $\delta(s) A_{s s'}=\delta(s') A_{s' s}$ and, if $s$ and $s'$ are conjugate, then\footnote{In the introduction, $A$ arises from a matrix $B$ defining a cluster algebra.
It may appear that requiring $\delta(s)=\delta(s')$ for $s$ conjugate to $s'$ places additional constraints on $B$.
However, this condition on $\delta$ holds automatically when $A$ is crystallographic, as explained in \cite[Section~2.3]{typefree}.}
 $\delta(s)=\delta(s')$. 
 
Let $V$ be a real vector space with basis $\set{\alpha_s:s\in S}$ (the \newword{simple roots}).
Let $s\in S$ act on $\alpha_{s'}$ by $s(\alpha_{s'})=\alpha_{s'}-A_{ss'} \alpha_s$. 
Vectors of the form $w\alpha_s$, for $s\in S$ and $w\in W$, are called \newword{roots}\footnote{In some contexts, these are called \newword{real roots}.}.  
The collection of all roots is the \newword{root system} associated to $A$.
The \newword{positive roots} are the roots which are in the positive linear span of the simple roots.
Each positive root has a unique expression as a positive combination of simple roots.
There is a bijection $t\mapsto\beta_t$ between the reflections $T$ in~$W$ and the positive roots.
Under this bijection, $\beta_s=\alpha_s$ and $w \alpha_s = \pm \beta_{wsw^{-1}}$. 

Let $\alpha_s\ck= \delta(s)^{-1} \alpha_s$.
The set $\set{\alpha_s\ck:s\in S}$ is the set of \newword{simple co-roots}.
The action of~$W$ on simple co-roots is $s(\alpha\ck_{s'})=\alpha\ck_{s'}-A_{s's} \alpha\ck_s$.
Let $K$ be the bilinear form on $V$ given by $K(\alpha\ck_{s}, \alpha_{s'})=A_{s s'}$. 
The form $K$ is symmetric because $K(\alpha_s, \alpha_{s'})= \delta(s) K(\alpha\ck_s, \alpha_{s'}) = \delta(s) A_{s s'} = \delta(s') A_{s' s} = K(\alpha_{s'}, \alpha_s)$.
The action of~$W$ preserves $K$.
We define $\beta_t\ck = (2/K(\beta_t, \beta_t)) \beta_t$. 
If $t = w s w^{-1}$, then $\beta\ck_t = \delta(s)^{-1} \beta_t$.
The action of $t$ on $V$ is by the relation $t \cdot x=x - K(\beta_t\ck, x) \beta_t=x - K(x, \beta_t) \beta\ck_t$.

A reflection $t \in T$ is an inversion of an element $w\in W$ if and only if $w^{-1} \beta_t$ is a negative root.
A simple generator $s\in S$ acts on a positive root $\beta_t$ by $s\beta_t=\beta_{sts}$ if $t\neq s$;
the action of $s$ on $\beta_s=\alpha_s$ is $s\alpha_s=-\alpha_s$. 

The following lemma is a restatement of the second Proposition of~\cite{Pilk}.
\begin{lemma} \label{PilkLemma}
 Let $I$ be a finite subset of $T$. Then the following are equivalent:
\begin{enumerate}
 \item[(i)] There is an element~$w$ of~$W$ such that $I=\inv(w)$.
\item[(ii)] If $r$, $s$ and $t$ are reflections in~$W$, with $\beta_{s}$ in the positive span of $\beta_r$ and $\beta_t$, then $I \cap \{ r,s,t \}\neq\set{s}$ and $I \cap \{ r,s,t \}\neq\set{r,t}$.
\end{enumerate}
\end{lemma}

We now review the theory of antimatroids; our reference is~\cite{antimatroid}.
Let $E$ be a finite set and $\LL$ be a collection of subsets of $E$. The pair $(E, \LL)$ is an antimatroid if it obeys the following axioms:\footnote{The reference~\cite{antimatroid} adds the following additional axiom: if $X \in \LL$, $X \neq \emptyset$, then there exists $x \in X$ such that $X \setminus \{ x \} \in \LL$. 
However, Lemma~\ref{BEZ Lemma} shows in particular that axioms (1) and (2) imply a condition numbered $(2')$.
Setting $Y=\emptyset$ and $Z=X$ in condition $(2')$, we easily see that the additional axiom of~\cite{antimatroid} follows from (1) and (2).} 
\begin{enumerate}
\item[$(1)$] $\emptyset \in \LL$.
\item[$(2)$] If $Y \in \LL$ and $Z \in \LL$ such that $Z \not \subseteq Y$, then there is an $x \in (Z \setminus Y)$ such that $Y \cup \{ x \} \in \LL$.
\end{enumerate}

\begin{prop}\label{anti max}
If $(E, \LL)$ is an antimatroid, then $\LL$ has a unique maximal element with respect to containment.
\end{prop}

\begin{proof}
By axiom $(1)$, $\LL$ is nonempty, so it has at least one maximal element. 
Suppose that $Y$ and $Z$ are both maximal elements of $\LL$. 
Since $Z$ is maximal, it is not contained in $Y$.
Now, axiom $(2)$ implies that $Y$ is not maximal, a contradiction.
\end{proof}

The next lemma and its proof are modeled after~\cite[Lemma~2.1]{BEZ}: 

\begin{lemma} \label{BEZ Lemma}
Let $E$ be a finite set and $\LL$ a collection of subsets of $E$. Then $\LL$ is an antimatroid if and only if $\LL$ obeys the following conditions.
\begin{enumerate}
\item[$(1)$] $\emptyset \in \LL$.
\item[$(2')$] For any $Y$ and $Z \in \LL$, with $Y \subseteq Z$, there is a chain $Y=X_0 \subset X_1 \subset \cdots \subset X_l =Z$ with every $X_i \in \LL$ and $\# X_{i+1} = \# X_{i} +1$.
\item[$(3')$] Let $X$ be in $\LL$ and let $y$ and $z$ be in $E \setminus X$ such that $X \cup \{ y \}$ and $X \cup \{z \}$ are in $\LL$. 
Then $X \cup \{ y,z \}$ is in $\LL$.
\end{enumerate}
\end{lemma}

\begin{proof}
First, we show that, if $(E, \LL)$ is an antimatroid, then $(E, \LL)$ obeys conditions $(2')$ and $(3')$.
For condition~$(2')$, we construct the $X_i$ inductively: 
Take $X_0$ to be~$Y$. 
If $X_i\neq Z$ then we apply axiom~$(2)$ to the pair $Z \not\subseteq X_i$ and set $X_{i+1} = X_i \cup \{ x \}$.
For condition~$(3')$, apply axiom~$(2)$ with $Y=X \cup \{ y \}$ and $Z = X \cup \{ z \}$.

Now we assume conditions~$(1)$, $(2')$ and~$(3')$ and show axiom~$(2)$. 
Let $X$ be an element of $\LL$ which is maximal subject to the condition that $X \subseteq Y \cap Z$. By condition~$(1)$, such an $X$ exists and, as $Z \not \subseteq Y$, we know that $X \subsetneq Z$.
Using condition~$(2')$, let $X = W_0 \subset W_1 \subset \cdots \subset W_l =Z$ be a chain from $X$ to $Z$ and let $W_1 = X \cup \{ x \}$. 
We now show that $x$ has the desired property.
By the maximality of $X$, we know that $x \not \in Y$.
Use condition~$(2')$ again to construct a chain $X = X_0 \subset X_1 \subset \cdots \subset X_r = Y$ from $X$ to $Y$. 
We will show by induction on $i$ that $X_i \cup \{ x \}$ is in $\LL$.
For $i=0$, this is the hypothesis that $W_1 \in \LL$.
For larger $i$, apply condition~$(3')$ to the set $X_{i-1}$, the unique element of $X_{i} \setminus X_{i-1}$, and the element~$x$. 
\end{proof}

For the remainder of the section, we fix~$W$,~$w$ and~$\Omega$, and we omit these from the notation where it does not cause confusion. 
Thus we write $L$ for the set $L(w,\Omega)$ of subsets~$J$ of $S$ such that~$J$ is $\Omega$-acyclic and $c(\Omega,J) \leq w$. 
We now turn to verifying conditions $(1)$, $(2')$ and $(3')$ for the pair $(S,L)$. 
Condition~$(1)$ is immediate.

\begin{lemma} \label{tea room}
Let $J_1$ and $J_2 \in L$.
Suppose that $J_1 \cup J_2$ is $\Omega$-acyclic and $\Omega|_{J_1 \cup J_2}$ has a linear extension $(q_1, q_2, \ldots, q_k, r, s_1, s_2, \ldots, s_l)$, where $J_1$ is $\{ q_1, q_2, \ldots, q_k, r \}$ and $J_2$ is $\{q_1, q_2, \ldots, q_k, s_1, s_2, \ldots, s_l \}$. 
Then $J_1 \cup J_2$ is in~$L$.
\end{lemma}

\begin{proof}
Since $J_1\in L$, we have $q_1 \cdots q_k\le q_1 \cdots q_kr=c(\Omega,J_1)\le w$.  
Similarly, because $J_2 \in L$, we know that $q_1\cdots q_ks_1\cdots s_l\le w$.
Defining $u$ so that $w=q_1 \cdots q_k u$, repeated applications of Lemma~\ref{AboveBelowLite} imply that $r\le u$ and also that $s_1\cdots s_l\le u$.

Define $t_1=s_1$, $t_2=s_1 s_2 s_1$, $t_3=s_1 s_2 s_3 s_2 s_1$ and so forth. 
The $t_i$ are inversions of $s_1\cdots s_l$, and thus they are inversions of $u$. 
Each $\beta_{t_i}$ is in the positive linear span of the simple roots $\set{\alpha_{s_j}:j=1,2,\ldots,l}$. 
None of these simple roots is $\alpha_r$, and since off-diagonal entries of $A$ are nonpositive, we have $K(\alpha_r\ck,\beta_{t_i})\le 0$.
So the positive root $\beta_{rt_ir}=r\beta_{t_i}=\beta_{t_i}-K(\alpha_r\ck,\beta_{t_i})\alpha_r$ is in the positive linear span of $\beta_r$ and $\beta_{t_i}$. 
Since~$t_i$ is an inversion of $u$, and $r$ is as well, we deduce by Lemma~\ref{PilkLemma} that $r t_i r$ is also an inversion of $u$. 
So $r$, $rt_1 r$, $r t_2 r$, \dots, and $r t_l r$ are inversions of $u$.
But $\inv(rs_1\cdots s_l)=\set{r,rt_1 r,r t_2 r, \dots,r t_l r}$, so $u\ge rs_1\cdots s_l$.
Applying Lemma~\ref{AboveBelowLite} repeatedly, we conclude that $w \geq (q_1 q_2 \cdots q_k) r  (s_1 \cdots s_l)=c(\Omega,J_1 \cup J_2)$.   
\end{proof}

We now establish condition $(2')$ for the pair $(S,L)$. 

\begin{lemma} \label{L Is Graded}
Let $I \subset J$ be two elements of $L$. Then there exists a chain $I = K_0 \subseteq K_1 \subseteq \ldots \subseteq K_l = J$ with each $K_i \in L$ and $\# K_{i+1} = \# K_i +1$.
\end{lemma}

\begin{proof}
It is enough to find an element $I'$ of $L$, of cardinality $\# I+1$, with $I \subset  I' \subseteq J$. 
Let $(y_1, y_2, \cdots y_j)$ be a linear extension of $\Omega|_J$. 
Let $y_{a}$ be the first entry of~$(y_1, y_2, \cdots y_j)$ which is not in $I$. 
So $w \geq c(\Omega,J) \geq y_1 y_2 \cdots y_{a-1} y_{a}$. 
Applying Lemma~\ref{tea room} to $(y_1, y_2, \cdots y_{a})$ and $I$, we conclude that $I \cup \{ y_1, y_2, \cdots y_{a} \}=I \cup \{ y_a \}$ is in $L$. 
Taking $I \cup \{ y_a \}$ for $I'$, we have achieved our goal.
\end{proof}

We now prepare to prove that $(S,L)$ satisfies condition~$(3')$.

\begin{lemma} \label{Inversion Formula} 
Let~$J$ be $\Omega$-acyclic and let $(s_1, s_2, \ldots, s_k)$ be a linear extension of $\Omega|_J$. 
Set $t=s_1 s_2 \cdots s_k \cdots s_2 s_1$. 
Then
\begin{equation}
\beta_t=\sum_{(r_1, r_2, \ldots,  r_j)} (- A_{r_j r_{j-1}}) \cdots (- A_{r_3 r_2}) (- A_{r_2 r_1}) \alpha_{r_1} 
\label{Inversion Expansion} \end{equation}
where the sum runs over all directed paths $r_1 \from r_2 \from \cdots \from r_j$ in $\Gamma \cap J$ with $r_j=s_k$. 
\end{lemma}

\begin{proof}
By a simple inductive argument, 
\[\beta_t=\sum_{(r_1, r_2, \ldots,  r_j)} (- A_{r_j r_{j-1}}) \cdots (- A_{r_3 r_2}) (- A_{r_2 r_1}) \alpha_{r_1},\]
where the summation runs over all subsequences of $(s_1, s_2, \ldots, s_k)$ ending in $s_k$. 
If there is no edge of $\Gamma$ between $r_i$ and $r_{i+1}$ then $(- A_{r_{i+1} r_i})=0$ so in fact we can restrict the summation to all 
subsequences which are also the vertices of a path through $\Gamma$. 
Since $(s_1, s_2, \ldots, s_k)$ is a linear extension of $\Omega|_J$, we sum over all directed paths $r_1\from r_2 \from \cdots \from r_j$ with $r_j=s_k$.
\end{proof}

\begin{lemma} \label{coeff at least 1}
Suppose $A$ is symmetric or crystallographic.
Let~$J$ be $\Omega$-acyclic and let $(s_1, s_2, \ldots, s_k)$ be a linear extension of $\Omega|_J$. 
Set $t=s_1 s_2 \cdots s_k \cdots s_2 s_1$. 
If $r\in J$ has $r \leq_J s_k$ then $\alpha_r$ appears with coefficient at least $1$ in the simple root expansion of $\beta_t$.
\end{lemma}
\begin{proof}
Since $A$ is either symmetric or crystallographic, $A_{ij} \leq -1$ whenever $A_{ij} <0$. 
Thus in Lemma~\ref{Inversion Formula}, every coefficient $(- A_{r_j r_{j-1}}) \cdots (- A_{r_3 r_2}) (- A_{r_2 r_1})$ in the sum is at least one. 
If $r \geq_J s_k$ then there is a directed path from $r$ to $s_k$ through~$J$, so the coefficient of $\alpha_r$ in $\beta_t$ is at least one.
\end{proof}

\begin{lemma} \label{No Chains}
Let $P$ and $Q$ be disjoint, $\Omega$-acyclic subsets of $S$. 
Suppose there exists $p\in P$ and $q\in Q$ such that there is an oriented path from $p$ to $q$ within $P \cup \set {q}$ and an oriented path from $q$ to $p$ within $Q \cup \set{p}$. 
Then there is no element of~$W$ which is greater than both $c(\Omega,P)$ and $c(\Omega,Q)$.
\end{lemma}

\begin{proof}
The lemma is a purely combinatorial statement about~$W$, and in particular does not depend on the choice of $A$.  
Thus, to prove the lemma, we are free to choose $A$ to be symmetric, so that we can apply Lemma~\ref{coeff at least 1}.
Furthermore, for $A$ symmetric, each root equals the corresponding co-root, and $A$ is the matrix of the bilinear form~$K$.

Let $(p_1,\cdots,p_k)$ be a linear extension of $\Omega|_P$ and let $(q_1,\cdots,q_n)$ be a linear extension of $\Omega|_Q$.
The hypothesis of the lemma is that there exist $i$, $j$, $l$ and $m$ with $1\le i \leq j\le k$ and $1\le l \leq m\le n$ such that there is a directed path from $p_j$ to $p_i$ in $P$, followed by an edge $p_i \to q_m$, and, similarly a directed path from $q_m$ to $q_l$ in $Q$ followed by an edge $q_l \to p_j$. 
The reflection $t=p_1p_2\cdots p_j\cdots p_2p_1$ is an inversion of $c(\Omega,P)$ and the reflection $u=q_1q_2\cdots q_m\cdots q_2q_1$ is an inversion of $c(\Omega,Q)$.
To prove the lemma, it is enough to show that no element of~$W$ can have both $t$ and $u$ in its inversion set. 

The positive root $\beta_t$ is a positive linear combination of simple roots $\set{\alpha_s:s\in P}$.
By Lemma~\ref{coeff at least 1},  $\alpha_{p_i}$ and $\alpha_{p_j}$ both appear with coefficient at least $1$ in $\beta_t$.
Similarly, $\beta_u$ is a positive linear combination of $\set{\alpha_s:s\in Q}$ in which $\alpha_{q_l}$ and $\alpha_{q_m}$ both appear with coefficient at least $1$.

Since $P$ and $Q$ are disjoint, we have $A_{rs} \leq 0$ for any $r \in P$ and $s \in Q$.
Also $K(\alpha_{p_j},\alpha_{q_l})\neq 0$, since $q_l\to p_j$, and thus $K(\alpha_{p_j},\alpha_{q_l})\le -1$.
Similarly, $K(\alpha_{p_i},\alpha_{q_m})\le -1$.
Thus
\[K( \beta_t, \beta_u ) \leq K( \alpha_{p_j}, \alpha_{q_l} ) +  K( \alpha_{p_i}, \alpha_{q_m} ) \leq -2.\]
Now $t$ acts on $\beta_u$ by $t \cdot \beta_u=\beta_u - K(\beta_t\ck, \beta_u) \beta_t=\beta_u - K(\beta_t,\beta_u) \beta_t$, and $u$ acts on $\beta_t$ similarly.
Thus $t$ and $u$ generate a reflection subgroup of infinite order.
Therefore, there are infinitely many roots in the positive span of $\beta_t$ and $\beta_u$.
In particular, by Lemma~\ref{PilkLemma}, no element of~$W$ can have both $t$ and $u$ as inversions.
\end{proof}

We now complete the proof of Theorem~\ref{antimatroid} by showing that $(S, L)$ satisfies condition $(3')$.
So let~$w \in W$, let $I\in L$ and let $a,a'\in S\setminus I$ such that $J=I \cup \{ a \}$ and $J'=I \cup \{ a' \}$ are both in $L$.

Our first major goal is to establish that $J \cup J'$ is $\Omega$-acyclic. 
This part of the argument is illustrated in Figure~\ref{Poset}.
Let $I_1$ be the set of all elements of $I$ lying on directed paths from $a$ to $a'$, and let $I_2$ be the set of all elements of $I$ lying on directed paths from $a'$ to $a$. 
Once we show that $J \cup J'$ is $\Omega$-acylic, we will know that either $I_1$ or $I_2$ is empty, but we don't know this yet. 
However, it is easy to see that $I_1$ and $I_2$ are disjoint, as an element common to both would lie on a cycle in~$J$.

Set $U = \{ u \in I : u \not \geq_J a \mbox{ and } u \not \geq_{J'} a' \}$.
The reader may find it easiest to follow the proof by first considering the special case where $U$ is empty.
Note that $U$ is disjoint from $I_1$ and $I_2$.

\begin{figure}
\psfrag{a}{\LARGE$a$}
\psfrag{b}{\LARGE$a'$}
\psfrag{I}{\LARGE$I_2$}
\psfrag{J}{\LARGE$I_1$}
\psfrag{U}{\LARGE$U$}
\centerline{\scalebox{0.70}{\includegraphics{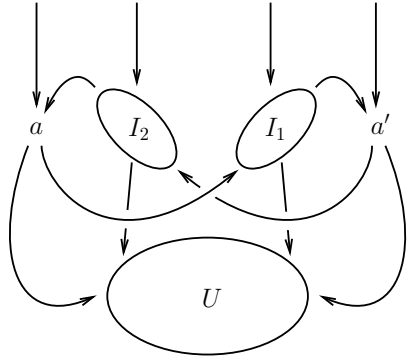}}}
\caption{The various subsets of $S$ occurring in the proof of $(3')$.} \label{Poset}
\end{figure}

Let $V_1 = U \cup I_1 \cup \{ a \}$. 
We claim that $V_1$ is a (lower) order ideal of $\Omega|_J$.
It is obvious that $U$ is an order ideal.
If $i \in I_1 \cup \{ a \}$, and $j <_{J} i$, then $j \in I_1$ if $j \geq_J a'$ and $j \in U$ otherwise. 
So $V_1$ is an order ideal of $\Omega|_J$ and we have $w \geq c(\Omega, J) \geq c(\Omega, V_1)$.
Moreover, since $U$ is an order ideal in $\Omega|_{V_1}$, we have $c(\Omega, V_1) = c(\Omega, U) c(\Omega, I_1 \cup \{a \})$ and thus $c(\Omega, U)^{-1} w \geq c(\Omega, I_1 \cup \{a \})$ by many applications of Lemma~\ref{AboveBelowLite}. 
Similarly, $c(\Omega, U)^{-1} w \geq c(\Omega, I_2 \cup \{ a' \})$.

Suppose (for the sake of contradiction) that $J\cup J'$ is not $\Omega$-acyclic. 
Since $J$ and $J'$ are $\Omega$-acyclic, there must exist both a directed path from $a$ to $a'$ and a directed path from $a'$ to $a$ in $J \cup J'$.
Applying Lemma~\ref{No Chains} with $P = I_1 \cup \{ a \}$, $p=a$, $Q = I_2 \cup \{ a' \}$ and $q = a'$, we deduce that no element of $W$ is greater than both $c(\Omega, P)$ and $c(\Omega, Q)$.
This contradicts the computations of the previous paragraph, so $J \cup J'$ is acyclic.

Choose a linear extension of $\Omega|_{J \cup J'}$. 
Without loss of generality, we may assume that $a$ precedes $a'$; let our linear ordering be $b_1$, $b_2$, \dots, $b_r$, $a$, $c_1$, $c_2$, \dots, $c_s$, $a'$, $d_1$, $d_2$, \dots, $d_t$. 
We can now apply Lemma~\ref{tea room} to the sequences $( b_1, b_2, \ldots, b_r, a )$ and $(b_1, b_2, \ldots, b_r, c_1, c_2, \ldots, c_s, a', d_1, d_2, \ldots, d_t)$ and deduce that $J \cup J'$ is in $L$.
This completes our proof of $(3')$.

\begin{remark}
It would be interesting to connect the antimatroid $(S,L(w,\Omega))$ to the antimatroids occurring in~\cite{Drew}.
\end{remark}

\section{$\Omega$-sortability and $\pidown^\Omega$}\label{main}
In this section, we define $\Omega$-sortable elements and the map $\pidown^\Omega$, review the definition of $c$-sortable elements and the map $\pidown^c$, and show how the $\Omega$- and $c$-versions of these concepts are related.
We then prove Proposition~\ref{Maximal Sortable} and Theorems~\ref{Sublattice} and~\ref{QuotientLattice}. 

For any $w\in W$, we appeal to Proposition~\ref{Maximal Coxeter} to inductively define a sequence of elements of~$W$ as follows:  Let $w_1=w$.
When $w_i$ has been defined, let $J_i=J(w_i,\Omega)$, and define $w_{i+1}=\left[c(\Omega,J_i)\right]^{-1}w_i$.
Since $\ell(w_{i+1})=\ell(w_i)-|J_i|$, the $J_i$ are empty for $i$ sufficiently large. 
It is clear that $J(v,\Omega)=\emptyset$ if and only if $v=e$, so we see that $w_i=e$ for $i$ sufficiently large. 
Thus, the infinite product $c(\Omega,J_1) c(\Omega,J_2) \cdots$ is defined, and equal to~$w$. 
For each $i$, fix a total order on $J_i$ that extends $\Omega|_{J_i}$.
In the expression $c(\Omega,J_1) c(\Omega,J_2) \cdots$, replace each $c(\Omega,J_i)$ by the reduced word for $c(\Omega,J_i)$ given by listing the elements of $J_i$ according to the total order.   
We thus obtain a reduced word called \newword{an $\Omega$-sorting word} for~$w$. 

We say that~$w$ is \newword{$\Omega$-sortable} if $J_1 \supseteq J_2 \supseteq J_3 \supseteq \cdots$. 
Observe that, if~$w$ is $\Omega$-sortable, then~$w$ automatically lies in $W_J$ for some $\Omega$-acyclic~$J$. 

We now review the definition of $c$-sortable elements in~$W$, where $c$ is a Coxeter element of~$W$. 
Fix a reduced word $s_1s_2\cdots s_n$ for~$c$ and define an infinite word 
\[(s_1\cdots s_n)^\infty=s_1s_2\cdots s_n|s_1s_2\cdots s_n|s_1s_2\cdots s_n|\ldots\]
The symbols ``$|$'' serve only to mark the boundaries between repetitions of the word $s_1s_2\cdots s_n$.
For each $w\in W$, the \newword{$(s_1\cdots s_n)$-sorting word} for $w\in W$ is the lexicographically first (as a sequence of positions in $(s_1\cdots s_n)^\infty$) subword of $(s_1\cdots s_n)^\infty$ that is a reduced word for~$w$. 
The $(s_1\cdots s_n)$-sorting word defines a sequence of subsets of~$S$:
Each subset is the set of letters of the $(s_1\cdots s_n)$-sorting word occurring between adjacent dividers. 

A $(s_1\cdots s_n)$-sorting word for~$w$ is also called a \newword{$c$-sorting word} for~$w$.
Thus there are typically several $c$-sorting words for~$w$, but exactly one $(s_1\cdots s_n)$-sorting word for~$w$ for each reduced word $s_1s_2\cdots s_n$ for $c$.
Each $c$-sorting word for~$w$ defines the same sequence of subsets.
A \newword{\mbox{$c$-sortable element}} of~$W$ is an element whose a $c$-sorting word defines a sequence of subsets which is weakly decreasing under inclusion.

\begin{remark}
Let~$w$ be an element of~$W$. We define $F(w, \Omega)$ to be the generating function $\sum x_1^{|J_1|} x_2^{|J_2|} \cdots x_r^{|J_r|}$, where the sum is over all length-additive factorizations $w = c(\Omega,J_1) c(\Omega,J_2) \cdots c(\Omega,J_r)$.   
(It is permitted that some $J_i$ be empty, and $r$ is permitted to vary.)
If $W$ is of type $A_n$, and $\Omega$ is oriented as $1 \to 2 \to \cdots \to n$, this is the Stanley symmetric function~\cite{Sta}, as shown in \cite[Proposition~5]{Lam}.
If $W$ is of type $\tilde{A}_n$, and $\Omega$ is the cyclic orientation, this is (essentially by definition) Lam's affine generalization of the Stanley symmetric functions.
The $c$- (respectively $\Omega$)-sorting word for $w$ corresponds to the unique dominant monomial constructed in~\cite[Section 4]{Sta} (respectively,~\cite[Theorem 13]{Lam}).
It would be interesting to see whether something could be said about $F(w,\Omega)$ for other groups and for other orientations of the diagrams.  
\end{remark}

If~$\Omega$ is acyclic, then $\Omega$-sortability coincides with $c(\Omega)$-sortability.
To understand why, it is enough to prove the following proposition.

\begin{prop}\label{SameAsAcyclic} 
If the orientation~$\Omega$ is acylic, then any $c(\Omega)$-sorting word for $w\in W$ is an $\Omega$-sorting word for $w$.
\end{prop}
\begin{proof}
Let $J_1,J_2,\ldots$ be the sequence of subsets of $S$ arising in the definition of the $\Omega$-sorting word for~$w$.
Fix a reduced word $s_1\cdots s_n$ for $c$, and let $I_1,I_2,\ldots$ be the sequence of subsets arising from the definition of the $(s_1\cdots s_n)$-sorting word for~$w$.
The content of the proposition is that these two sequences coincide.  
The definition of $I_1$ assures that $c(\Omega,I_1)\le w$, so $I_1\subseteq J_1$, by the definition of $J_1$.
If $I_1\subsetneq J_1$ then any word starting with a reduced word for $c(\Omega,J_1)$ is a lexicographically earlier subword of $(s_1\cdots s_n)^\infty$ than the $(s_1\cdots s_n)$-sorting word for~$w$, which omits the letters in $J_1\setminus I_1$.
Thus $I_1=J_1$.

Now $J_2,J_3,\ldots$ and $I_2,I_3,\ldots$ are the sequences arising from the $\Omega$- or $c$-sorting word for $c(\Omega,J_1)^{-1}w$.
By induction on the length of~$w$, these sequences coincide.
\end{proof}

The next proposition says that, when~$\Omega$ is not acyclic, the notions of $\Omega$-sortability and $c$-sortability are related as described in the introduction.
\begin{prop}\label{ParaAndAcyclic}
Let $w\in W$.
Then~$w$ is $\Omega$-sortable if and only if~$w$ is a $c(\Omega,J(w,\Omega))$-sortable element of $W_{J(w,\Omega)}$.
\end{prop}
\begin{proof}
If~$w$ is $\Omega$-sortable, then every letter in its $\Omega$-sorting word is contained in $J(w,\Omega)$, and thus $w\in W_{J(w,\Omega)}$.
Furthermore,~$w$ is $\Omega|_J$-sortable and thus $c(\Omega,J(w,\Omega))$-sortable by Proposition~\ref{SameAsAcyclic}.
The argument is easily reversed.
\end{proof}

We now give the recursive definition of $\pidown^\Omega$.
For any $w \in W$, set $J:=J(w,\Omega)$ and define $\pidown^{\Omega}(w)=c(\Omega,J) \pidown^{\Omega}\bigl[(c(\Omega,J))^{-1}w_J\bigr]$.
Setting $\pidown^{\Omega}(e)=e$, this recursion terminates.

Proposition~\ref{Maximal Sortable} is the assertion that $\pidown^{\Omega}(w)$ is the unique maximal $\Omega$-sortable element below~$w$ in the weak order.
In order to prove Proposition~\ref{Maximal Sortable}, we will appeal to the acyclic case of Theorem~\ref{o p}, which was proved as \cite[Theorem~6.1]{typefree}.
The latter theorem is a statement about a map $\pidown^c$, whose definition we now review.

Fix a reduced word $s_1 s_2 \cdots s_n$ for $c$ and let $w \in W$. 
Let $\Omega$ be the corresponding acyclic orientation of $\Gamma$.
The definition of $\pidown^c(w)$ in \cite[Section~6]{typefree} was inductive, stepping through one letter of $(s_1 s_2 \cdots s_n)^{\infty}$ at a time.
For our present purposes, it is easier to perform each $n$ steps at once. 
The definition from~\cite{typefree} is then equivalent to the following:  
Setting $J_0 = \emptyset$, we will successively construct subsets $J_1$, $J_2$, \ldots, $J_n$ with $J_i \subseteq [i]$.
If $w \geq c(\Omega,J_{i-1}) s_i$, then $J_i = J_{i-1} \cup \{ i \}$; otherwise, $J_i = J_{i-1}$. 
Set $J = J_n$. Then $\pidown^c(w) = c(\Omega,J) \cdot  \pidown^{c}\left( (c(\Omega,J)^{-1} w)_J \right)$. 

The base case of the inductive proof of Proposition~\ref{SameAsAcyclic} establishes that $J_n=J(w,\Omega)$.
Thus $\pidown^{c(\Omega)}$ coincides with $\pidown^\Omega$ when~$\Omega$ is acyclic.
Furthermore, when~$\Omega$ is not necessarily acyclic, $\pidown^{\Omega}(w)=\pidown^{c(\Omega,J(w,\Omega))}(w_{J(w, \Omega)})$.

\begin{proof}[Proof of Proposition~\ref{Maximal Sortable}]
Let $w\in W$, abbreviate $J(w, \Omega)$ to~$J$ and abbreviate $c(\Omega,J)$ to $c$. 
We need to show that $\pidown^c(w_J)$ is the unique maximal $\Omega$-sortable element below~$w$ in the weak order.
We have $w \geq w_J$ and, by the acyclic case of Theorem~\ref{o p}, $w_J \geq \pidown^c(w_J)$. 
Also, $\pidown^c(w_J)$ is $c$-sortable, and hence $\Omega$-sortable by Proposition~\ref{ParaAndAcyclic}.  
We now must check that, if $v$ is $\Omega$-sortable and $v \leq w$, then $v \leq \pidown^c(w_J)$.
Since $v$ is $\Omega$-sortable, we deduce that $v \in W_{J(v, \Omega)}$.
But $w \geq v \geq c(\Omega,J(v, \Omega))$, so $J(w, \Omega) \supseteq J(v, \Omega)$ and $v \in W_J$.
Now Proposition~\ref{wJ o p} says that $w_J \geq v_J=v$ and, appealing again to the acyclic case of Theorem~\ref{o p}, $\pidown^c(w_J) \geq \pidown^c(v) =v$. 
\end{proof}

Now that we have proven Proposition~\ref{Maximal Sortable}, we also have, as corollaries, Theorem~\ref{o p} and Propositions~\ref{idem} and~\ref{downward}.
We also obtain the following proposition by reduction to the acyclic case, which was proven as \cite[Proposition~3.13]{typefree}.

\begin{prop} \label{Sortable restriction}
Let~$v$ be an $\Omega$-sortable element of~$W$ and let $I$ be any subset of $S$. Then $v_I$ is $\Omega|_I$-sortable.
\end{prop}

\begin{proof}
Set $J = J(v,\Omega)$. 
So $v \in W_J$, $\Omega|_J$ is acyclic, and~$v$ is $\Omega|_J$-sortable. 
Since $v_I = v_{I \cap J}$, the acyclic case of the proposition says that $v_I$ is $\Omega|_{I \cap J}$-sortable, so it is $\Omega|_I$-sortable.
\end{proof}
 
A more difficult reduction to the acyclic case is needed to prove Theorem~\ref{Sublattice}.
The acyclic case was proven as \cite[Theorem~7.1]{typefree}.

\begin{proof}[Proof of Theorem~\ref{Sublattice}] 
First, suppose that $A$ is a nonempty set of $\Omega$-sortable elements. 
By Proposition~\ref{ParaAndAcyclic}, every element $a$ of $A$ lies in a parabolic subgroup $W_{J(a, \Omega)}$ where $J(a, \Omega)$ is acyclic. 
Let $J=\bigcap_{a \in A} J(a, \Omega)$.  
Since each $W_{J(a,\Omega)}$ is a lower order ideal, the element $\Meet A$ lies in $W_J$.
Thus $\Meet A=(\Meet A)_J$, which equals $\Meet_{a \in A} a_J$ by Proposition~\ref{para hom}.   
By Proposition~\ref{Sortable restriction}, every $a_J$ is $c(\Omega,J)$-sortable so, by the acyclic case, $\Meet_{a \in A} a_J$ is also $c(\Omega, J)$-sortable and thus $\Omega$-sortable by Proposition~\ref{ParaAndAcyclic}. 

Now, suppose $A$ is a set of $\Omega$-sortable elements such that $\Join A$ exists.
Since $A$ is contained in the interval below $\Join A$, in particular $A$ is finite.
Thus it is enough to consider the case where $A$ only has two elements, $u$ and $v$.
Let $I=J(u, \Omega)$ and let $J=J(v,\Omega)$. 
Now $u\ge c(\Omega,I)$ and $v\ge c(\Omega,J)$.

We will show that $J(u \join v, \Omega) = I \cup J$. 
As $u \join v \geq u \geq c(\Omega,I)$, Proposition~\ref{Maximal Coxeter} tells us that $J(u \join v, \Omega) \supseteq I$.  
By similar logic, $J(u \join v, \Omega) \supseteq J$, so $J(u \join v, \Omega) \supseteq I \cup J$.
On the other hand,  $u\in W_I$ and $v\in W_J$, so $u_{I\cup J}=u$ and $v_{I\cup J}=v$.
By Proposition~\ref{para hom}, $u\join v=u_{I\cup J}\join v_{I\cup J}=(u\join v)_{I\cup J}$, so $u\join v\in W_{I\cup J}$ and $J(u \join v, \Omega) \subseteq I \cup J$..
We now know that $J(u \join v, \Omega)= I \cup J$.
In particular, $I \cup J$ is $\Omega$-acyclic.

Now, $u$ and $v$ are both $\Omega|_{I \cup J}$-sortable elements of $W_{I \cup J}$. By the acyclic case, we deduce that $u \join v$ is $\Omega|_{I \cup J}$-sortable, and thus $\Omega$-sortable.
\end{proof}

Finally, we prove Theorem~\ref{QuotientLattice}, which states that $\pidown^\Omega$ factors over meets and joins.
We will appeal to the acyclic case of Theorem~\ref{QuotientLattice}, proved as \cite[Theorem~7.3]{typefree}.

\begin{proof}[Proof of Theorem~\ref{QuotientLattice}]
The proof of the assertion about meets exactly follows the argument in \cite[Theorem~7.3]{typefree} for the acyclic case, except that
 \cite[Theorem~6.1]{typefree} and \cite[Theorem~7.1]{typefree} are replaced by Theorems~\ref{o p} and~\ref{Sublattice}.

To prove the assertion about joins, set $J = J(\Join A,\Omega)$. 
Now $\pidown^{\Omega} \left( \Join A \right) = \pidown^{\Omega} \left( \left( \Join A \right)_J \right)$ which, by Proposition~\ref{para hom}, is $\pidown^{\Omega} \left( \Join A_J \right)$.   
The latter equals $\pidown^{\Omega|_J} \left( \Join A_J \right)$ which, by the acyclic case of the theorem, equals $\Join \pidown^{\Omega|_J}(A_J)$. 
Now, for each $a \in A$, we have $a \leq \Join A$, so $J(a, \Omega) \subseteq J$. 
Therefore, for each $a \in A$, we have $\pidown^{\Omega|_J}(a_J)=\pidown^{\Omega}(a)$. 
Thus $\Join \pidown^{\Omega|_J}(A_J)=\Join\pidown^{\Omega}(a)$ and, stringing together all of the equalities we have proved, we obtain the result.
\end{proof}

\section{The fibers of $\pidown^c$} \label{sec fibers} 
In this section, we describe the fibers of $\pidown^\Omega$ in terms of polyhedral geometry.
We begin by reviewing the analogous description in the acyclic case.

The \newword{dominant chamber} is the full-dimensional simplicial cone
\[D=\bigcap_{s \in S} \set{x^*\in V^*: \br{x^*,\alpha_s}\ge 0}\]
in $V^*$.
The map $w\mapsto wD$ takes~$W$ bijectively to a collection of $n$-dimensional cones with pairwise disjoint interiors.

In \cite[Section~5]{typefree}, a linearly independent set $C_c(v)$ of roots is defined recursively for each $c$-sortable element $v$.
More specifically, we define $n$ linearly independent roots $C_c^r(v)$, one for each $r \in S$, and set $C_c(v)=\{ C_c^r(v): \ r \in S \}$.
The set $\Cone_c(v)$, defined by $\bigcap_{r \in S} \left\{ x^*\in V^* : \langle x^*, C_c^r(v) \rangle \geq 0 \right\}$, is thus a full-dimensional, simplicial, pointed cone in $V^*$.
By \cite[Theorem~6.3]{typefree}, these cones characterize the fibers of $\pidown^c$ in the sense that $\pidown^c(w)=v$ if and only if $wD$ lies in $\Cone_c(v)$.

To generalize $C_c(v)$ to the cyclic setting, we imitate a non-recursive characterization of $C_c^r(v)$ which appears as~\cite[Proposition~5.1]{typefree}.
Fix a reduced word $s_1\cdots s_n$ for $c$, and let $a_1 a_2 \cdots a_k$ be the $(s_1\cdots s_n)$-sorting word for $v$.
Recall from Section~\ref{main} that the $(s_1 s_2 \cdots s_n)$-sorting word for $v$ is the lexicographically leftmost subword of $(s_1\cdots s_n)^\infty$ that is a reduced word for $v$.
In particular, $a_1 a_2 \cdots a_k$ is associated to a specific set of positions in $(s_1\cdots s_n)^\infty$.
For each $r\in S$, consider the first occurrence of $r$ in $(s_1\cdots s_n)^\infty$ that is \textbf{not} in a position occupied by $a_1 a_2 \cdots a_k$.
Let this occurrence of $r$ be between $a_i$ and $a_{i+1}$; we define $C_c^r(v) := a_1 a_2 \cdots a_i \alpha_r$.  

We now make a definition for the case where~$\Omega$ may contain cycles. 
Let $v$ be $\Omega$-sortable and let $J = J(\Omega, v)$. 
If $J \cup \{ r \}$ is $\Omega$-acyclic, define $C_{\Omega}^r(v)$ to be $C_{c(\Omega, J \cup \{ r \})}^r(v)$. 
If $J \cup \{ r \}$ is not $\Omega$-acyclic, then $C_{\Omega}^r(v)$ is undefined. 
Set $\Cone_{\Omega}(v) = \bigcap_{r} \left\{x^*\in V^* : \langle x^*, C_{\Omega}^r(v) \rangle \geq 0 \right\}$, where the intersection is over those $r$ such that $C_\Omega^r(v)$ is defined.

\begin{theorem}\label{pidown fibers}
Let $w \in W$.
Then $\pidown^{\Omega}(w)=v$ if and only if $wD \subseteq \Cone_{\Omega}(v)$. 
\end{theorem}

Once again, the proof draws on the acyclic case, which was proved as \cite[Theorem~6.3]{typefree}.
The proof also requires two facts about the polyhedral geometry of Coxeter groups, which we now provide.
First, if $t$ is any reflection of~$W$, then $wD\subseteq\set{x^*\in V^* : \br{x^*, \beta_t}\le 0}$ if and only if $t$ is an inversion of~$w$.
Second, for any subset $J\subseteq S$, define
\[D_J=\bigcap_{s \in J} \set{x^*\in V^*: \br{x^*,\alpha_s}\ge 0}.\]
There is an inclusion $wD\subseteq w_JD_J$ for any $w\in W$.
For details, see \cite[Section~2.4]{typefree}, but notice that the set $D_J$ defined here corresponds to $P_J^{-1}(D_J)$ in the notation of~\cite{typefree}.
The map $P_J$ is a certain projection map which we do not need here.

\begin{proof}[Proof of Theorem~\ref{pidown fibers}]
We continue the notation $J=J(\Omega,v)$.
First, suppose that $\pidown^{\Omega}(w)$ is $v$.  
We need to show that $wD\subseteq \set{x^*\in V^* : \br{x^*, C_{\Omega}^r(v)}\ge 0 }$ for all $r\in S$ such that $J \cup \{ r \}$ is $\Omega$-acyclic. 
For such an $r$, the element $\pidown^{c(\Omega, J \cup\set{r})}(w_{J \cup \{ r \}})$ coincides with $\pidown^{\Omega}(w)=v$.
By the acyclic case of the theorem, $w_{J\cup\set{r}} D_{J\cup\set{r}}$ is contained in $\bigl\{x^*\in V^* : \br{x^*, C_{c(\Omega,J\cup\set{r})}^r(v)}\ge 0 \bigr\}$.
But $C_{c(\Omega,J \cup \{ r \})}^r(v)$ coincides with $C_{\Omega}^r(v)$, so $wD\subseteq w_{J\cup\set{r}} D_{J\cup\set{r}}$ $\subseteq\set{x^*\in V^* : \br{x^*, C_{\Omega}^r(v)}\ge 0 }$.  

Now, suppose that $wD \subseteq \Cone_{\Omega}(v)$. 
We first note that $v \in W_J$ and, for $r \in J$, that $C_{\Omega}^r(v) = C_{\Omega|_J}^r(v)$. 
So $\Cone_{\Omega}(v) \subseteq \Cone_{\Omega|_{J}}(v)=\Cone_{c(\Omega, J)}(v)$ and thus $w D \subseteq  \Cone_{c(\Omega, J)}(v)$.  
Every cone of the form $u D_J$ is either completely contained in $\Cone_{c(\Omega, J)}(v)$ or has its interior disjoint from $\Cone_{c(\Omega, J)}(v)$.
We conclude that $w_JD_J\subseteq \Cone_{c(\Omega, J)}(v)$. 
Then $w_J D \subseteq w_JD_J\subseteq \Cone_{c(\Omega, J)}(v)$, so $\pidown^{c(\Omega, J)}(w_J) = v$ by the acyclic case of the theorem. 
Since $\pidown^\Omega(w)=\pidown^{c(\Omega, J(w,\Omega))}(w_{J(w,\Omega)})$, we can complete the proof by showing that $J(w,\Omega)=J$. 
Set $J'=J(w,\Omega)$.

Since $w\ge w_J\ge \pidown^{c(\Omega, J)}(w_J) = v$, it is immediate from the definition of $J(w,\Omega)$ that $J'\supseteq J$.
Suppose, for the sake of contradiction, that $J' \neq J$. 
By definition, $J'$ is $\Omega$-acyclic. 
Choose a linear extension $(a_1,a_2,\ldots,a_i,r,\ldots)$ of $\Omega|_{J'}$, where $r$ is the first element not in~$J$.
Then $C_{\Omega}^r(v)$ is the positive root $a_1 a_2 \cdots a_i \alpha_r$ and hence, by the assumption that $wD \subseteq \Cone_{\Omega}(v)$, we have $wD\subseteq\set{x^*\in V^* : \br{x^*, a_1 a_2 \cdots a_i \alpha_r}\ge 0}.$
On the other hand, $w \geq c(\Omega, J') \geq a_1 a_2 \cdots a_i r$, by the definition of $J'$. 
Thus $wD\subseteq\set{x^*\in V^* : \br{x^*, a_1 a_2 \cdots a_i \alpha_r}\le 0}$, because $a_1 a_2 \cdots a_i r$ is the positive root associated to an inversion of~$w$.
But $wD$ is a full-dimensional cone, and this contradiction establishes that $J=J'$.
\end{proof}

In the acyclic case, \cite[Theorem~9.1]{typefree} states that the cones $\Cone_c(v)$ (and their faces) form a fan in $\Tits(W)$.
Roughly, the assertion is that these cones fit together nicely within the Tits cone, but not necessarily everywhere. 
(See \cite[Section~9]{typefree} for the precise definition.)  
We observe that the proof in~\cite{typefree} also works without alteration in the more general setting, replacing \cite[Theorem~7.3]{typefree} by its generalization Theorem~\ref{QuotientLattice}.

We now describe the shortcomings of Theorem~\ref{pidown fibers} for the purposes of cluster algebras.
In the acyclic case, the cones $\Cone_c(v)$ correspond to clusters in the corresponding cluster algebra. 
More specifically,~\cite{clusters} establishes that the extreme rays of $\Cone_c(v)$ are spanned by the $\g$-vectors of the cluster variables; this is also shown in~\cite{YZ} for cluster algebras of finite type.
(One interprets the $\g$-vectors as coefficients of an expansion in the basis of fundamental weights.)
The cone $\Cone_c(v)$ has $|S|$ extreme rays because it is a pointed simplicial cone, or equivalently, because $C_c(v)$ is a set of $|S|=\dim(V)$ linearly independent vectors.

By contrast, the cone $\Cone_\Omega(v)$ may have fewer than $|S|$ defining hyperplanes, since $C_{\Omega}^r(v)$ undefined when $\Omega|_{J \cup \{ r \}}$ has a cycle.
In~\cite{clusters}, it is shown that each $\Omega$-sortable element $v$ corresponds to a cluster.
Thus, in order to fill in the cluster algebras picture, we need to define vectors $C_\Omega^r(v)$, in the cases we presently leave undefined, so as to turn $\Cone_\Omega(v)$ into a pointed simplicial cone with the right extreme rays.
This appears to be a hard problem, for reasons we now describe.

By computing $g$-vectors, we can determine what the missing values of $C_\Omega^r(v)$ should be.
However, we sometimes obtain that $C_\Omega^r(v)$ should not be a real root!
Consider the $B$-matrix $\left( \begin{smallmatrix} 0 & 1 & -1 \\ -1 & 0 & 1 \\ 2 & -2 & 0 \end{smallmatrix} \right)$. 
The corresponding Cartan matrix $A$ defines a hyperbolic Coxeter group\footnote{Although this Coxeter group is of wild type, the $B$-matrix is mutation equivalent to the finite type $B_3$ matrix.} of rank $3$.  
Call the simple generators $p$, $q$ and $r$ in the order of the rows/columns of $A$, and consider the $\Omega$-sortable element $v=qrq$. 
The roots $C_{\Omega}^q(v)$ and $C_{\Omega}^r(v)$ are defined, and equal to $- \alpha_q - 2 \alpha_r$ and $\alpha_r$ respectively. 
By calculating $\g$-vectors, one can check that $C_{\Omega}^p(v)$ should be $\alpha_p + \alpha_q + 2 \alpha_r$. 
This is an imaginary root! 
It would require a significant modification of the definition of $C_\Omega$ to output an imaginary root. 
It is easy to create a simply laced example with the same difficulty, by building a rank $4$ simply laced Coxeter group which folds to this example.

\section{Alignment} \label{Alignment sec}
The results of~\cite{typefree} make significant use of a skew-symmetric form $\omega_c$ on $V$ defined by setting $\omega_c(\alpha_r\ck, \alpha_s) = A_{rs}$ if $r\to s$.
The form $\omega_c$ provides, in particular, a characterization \cite[Proposition~3.11]{typefree} of $c$-sorting words for $c$-sortable elements and a characterization \cite[Theorem~4.2]{typefree} of inversion sets of $c$-sortable elements.
The two characterizations are as follows:
\begin{theorem} \label{AlignmentWord}
Let $c$ be a Coxeter element of~$W$.
Let $a_1 a_2 \cdots a_k$ be a reduced word for $w\in W$. 
Set $t_1 = a_1$, $t_2 = a_1 a_2 a_1$, \ldots, $t_k = a_1 a_2 \cdots a_k \cdots a_2 a_1$.
Then the following are equivalent:
\begin{enumerate}
 \item~$w$ is $c$-sortable and $a_1 a_2 \cdots a_n$ can be transformed into a $c$-sorting word for~$w$ by a sequence of transpositions of adjacent commuting letters.
\item For $i<j$, we have $\omega(\beta_{t_i}, \beta_{t_j}) \geq 0$, with strict inequality holding unless $t_i$ and $t_j$ commute.
\end{enumerate}
\end{theorem}

\begin{theorem}\label{AlignmentInversions}
Let $c$ be a Coxeter element of~$W$ and let $w\in W$.
Then the following are equivalent:
\begin{enumerate}
 \item~$w$ is $c$-sortable.
\item Whenever $r$, $s$ and $t$ are reflections in~$W$, with $\beta_{s}$ in the positive span of $\beta_r$ and $\beta_t$ and $\omega_c(\beta_r, \beta_t) \geq 0$, then $\inv(w) \cap \{ r,s,t \}$ is either $\emptyset$, $\{ r \}$, $\{ r, s \}$, $\{ r, s, t \}$ or $\{ t \}$.
\end{enumerate}
\end{theorem}

One can define an analogous skew-symmetric form on $V$ in the case of orientations with cycles. 
Define $\omega_{\Omega}$  by $\omega_{\Omega}(\alpha_r\ck, \alpha_s) = \pm A_{rs}$, where the positive sign is taken if $r \to_{\Omega} s$ and the negative sign if $s \to_{\Omega} r$. 
If there is no edge between $r$ and $s$ then $A_{rs}=0$, and $\omega_{\Omega}(\alpha_{r}\ck, \alpha_s)=0$.
The following is easily verified, by reduction to the acyclic case:
When $c$ is replaced by~$\Omega$ in either Theorem~\ref{AlignmentWord} or~\ref{AlignmentInversions}, the first condition still implies the second. 
Unfortunately, the reverse implications are no longer valid.
More precisely:

\begin{counterexample}\label{alignment problems}
There exists a Cartan matrix $A$, an orientation~$\Omega$ of $\Gamma$ and an element~$w$ with reduced word $a_1 a_2 \cdots a_k$ such that:
\begin{enumerate}
\item~$w$ is not $\Omega$-sortable; but
\item For $i<j$, we have $\omega_c(\beta_{t_i}, \beta_{t_j}) \geq 0$, with strict inequality holding unless $t_i$ and $t_j$ commute; and
\item Whenever $r$, $s$ and $t$ are reflections in~$W$, with $\beta_{s}$ in the positive span of $\beta_r$ and $\beta_t$ and $\omega_c(\beta_r, \beta_t) \geq 0$, then $\inv(w) \cap \{ r,s,t \}$ is either $\emptyset$, $\{ r \}$, $\{ r, s \}$, $\{ r, s, t \}$ or $\{ t \}$.
\end{enumerate}
\end{counterexample}

The third condition in Counterexample~\ref{alignment problems} may appear to be hard to check. 
Fortunately, it is redundant.

\begin{prop}
If $A$,~$\Omega$,~$w$ and $a_1a_2\cdots a_k$ are chosen so that condition (2) of Counterexample~\ref{alignment problems} holds, then condition (3) holds as well.
\end{prop}

\begin{proof}
In light of Lemma~\ref{PilkLemma}, we need only rule out the case where $\inv(w) \cap \{ r,s,t \}=\{ s,t \}$. 
Let $i$ and $j$ be such that $s = t_i$ and $t = t_j$.
Since $\omega_{\Omega}(\beta_s, \beta_t) \geq 0$, we have $i < j$.  
Set $w' = a_1 a_2 \cdots a_i$. Then $\inv(w') \cap \{ r,s,t \} = \{ s \}$, contradicting Lemma~\ref{PilkLemma}.
\end{proof}

Thus, to give a counter-example, we need only check conditions (1) and (2).
Consider a counter-example of rank $3$ with $B$-matrix  
$$\begin{pmatrix} 0 & 1 & -1 \\ -1 & 0 & 3 \\ 1 & -3 & 0 \end{pmatrix}$$
 with simple reflections $p$, $q$ and $r$.
Then $pqr$ is not $\Omega$-sortable, as its support is a cycle. 
But the corresponding inversion sequence is $p$, $pqp$, $pqrqp$ with roots 
\begin{alignat*}{20}
\beta_1 :=& \ \beta_p \ &=& \  &\alpha_p & & & & & \\
\beta_2 :=& \ \beta_{p q p} \ &=& \ &\alpha_p & {}+& \alpha_{q} & & &   \\
\beta_3 :=& \ \beta_{p q r q p} \  &=& \ 4 &\alpha_p & {}+& 3  \alpha_q & {}+& \alpha_r 
\end{alignat*}
We have  $\omega_{\Omega}(\beta_1, \beta_2) = 1$, $\omega_{\Omega}(\beta_1, \beta_3) =  2$, and $\omega_{\Omega}(\beta_2, \beta_3) = 1$.
All of these are positive, so this is a counterexample. 

\begin{remark} 
The definition of $\omega_{\Omega}$ depends not only on $\Omega$ and on the Coxeter group $W$, but also on the choice of a Cartan matrix.  
To illustrate the effect of this choice, consider a modification of the example above, with the entries $3$ and $-3$ replaced by $2$ and $-2$ respectively.
The Coxeter group $W$ is unchanged, $pqr$ is still not $\Omega$-sortable, and $\beta_1$ and $\beta_2$ are unchanged, while $\beta_3$ becomes $3\alpha_p+2\alpha_q+\alpha_r$.
We calculate $\omega_{\Omega}(\beta_1, \beta_2) = 1$, $\omega_{\Omega}(\beta_1, \beta_3) =  1$, and $\omega_{\Omega}(\beta_2, \beta_3) = 0$.
Since $pqp$ and $pqrqp$ do not commute, condition (2) of Counterexample~\ref{alignment problems} fails, and the modified example is not a counterexample.
\end{remark}

\begin{remark}
A preprint version of this paper proposed a different counter-example.  
We are grateful to the referee for pointing out that the earlier example was in error.  
\end{remark}

\end{document}